\documentclass[12pt,a4paper,reqno]{amsart}

\usepackage{graphicx}
\usepackage{amsfonts,amsthm,latexsym,amsmath,amssymb,amscd,epsfig}
\usepackage{graphics,graphicx}
\usepackage[all]{xy}
\usepackage{young}

\usepackage{tikz}
\usetikzlibrary{matrix}

\newtheorem{theorem}{Theorem}[section]
\newtheorem{lemma}[theorem]{Lemma}

\theoremstyle{definition}

\newtheorem{proposition}[theorem]{Proposition}
\newtheorem{corollary}[theorem]{Corollary}
\theoremstyle{remark}

\numberwithin{equation}{section}

\newcommand{\nc}{\newcommand}
\renewcommand{\frak}{\mathfrak}
\providecommand{\cal}{\mathcal}
\renewcommand{\bold}{\mathbf}

\newcommand{\ZZ}{{\mathbb Z}}

\nc \Ab{{\ensuremath{\bold A}}}
\nc \ab{{\ensuremath{\bold a}}}
\nc \bb{{\ensuremath{\bold b}}}
\nc \cb{{\ensuremath{\bold c}}}
\nc \Bb{{\ensuremath{\bold B}}}
\nc \Gb{{\ensuremath{\bold G}}}
\nc \Qb{{\ensuremath{\bold Q}}}
\nc \Rb{{\ensuremath{\bold R}}} \nc \Cb{{\ensuremath{\bold C}}} 
\nc \Eb{{\ensuremath{\bold E}}}
\nc \eb{{\ensuremath{\bold e}}}
\nc \Db{{\ensuremath{\bold D}}}
\nc \Fb{{\ensuremath{\bold F}}}
\nc \ib{{\ensuremath{\bold i}}}
\nc \jb{{\ensuremath{\bold j}}}
\nc \kb{{\ensuremath{\bold k}}}
\nc \nb{{\ensuremath{\bold n}}}
\nc \rb{{\ensuremath{\bold r}}}
\nc \Pb{{\ensuremath{\bold P}}}
\nc \pb{{\ensuremath{\bold p}}}
\nc \SPb{{\ensuremath{\bold {SP}}}}
\nc \Zb{{\ensuremath{\bold Z}}} 
\nc \zb{{\ensuremath{\bold z}}} 
\nc \gb{{\ensuremath{\bold g}}} 
\nc \fb{{\ensuremath{\bold f}}} 
\nc \ub{{\ensuremath{\bold u}}} 
\nc \vb{{\ensuremath{\bold v}}} 
\nc \yb{{\ensuremath{\bold y}}} 
\nc \xb{{\ensuremath{\bold x}}} 
\nc \xib{{\ensuremath{\bold \xi}}} 
\nc \Nb{{\ensuremath{\bold N}}} 
\nc \Hb{{\ensuremath{\bold H}}} 
\nc \wb{{\ensuremath{\bold w}}} 
\nc \Wb{{\ensuremath{\bold W}}} 
\nc \syz{{\mathbf {syz}}}
\nc \bnoll{{\ensuremath{\bold 0}}} 

\nc \mf{\frak m} \nc \mh{\hat{\m}} 
\nc \nf{\frak n}
\nc \Of{\frak O}
\nc \rf{\frak r}
\nc \mufr{{\mathbf \mu}}
\nc \hf{\frak h} 
\nc \qf{\frak q} 
\nc \bfr{\frak b} 
\nc \kfr{\frak k} 
\nc \pfr{\frak p} 
\nc \af{\frak a }
\nc \cf{\frak c }
\nc \sfr{\frak s} 
\nc \ufr{\frak u} 
\nc \g{\frak g} 
\nc \gA{\g_{\Ao}} 
\nc \lfr{\frak l}
\nc \afr{\frak a}
\nc \gfh{\hat {\frak g}}
\nc \gl{\frak { gl }}
\nc \Sl{\frak {sl}}
\nc \SU{\frak {SU}}
\nc{\Homf}{\frak{Hom}}

\newcommand{\on}{\operatorname}
\nc\hankel{\on {Hankel}}
\nc\row{\on {row\ }}
\nc\nullity{\on {nullity }}
\nc\col{\on {col\ }}
\nc\rowm{\on {Row \ }}
\nc\loc{\on {lc \ }}
\nc\nullo{\on {null\ }}
\nc\Nul{\on {Nul\ }}
\nc \Ann {\on {Ann }}
\nc \Ass {\on {Ass \ }}
\nc \Coker {\on {Coker}}
\nc \Co{\on C}
\nc \Homo{\on {Hom}}
\nc \Ker {\on {Ker}}
\nc \omod{\on {mod}}
\nc \No {\on N}
\nc \NN {\on {NN}}
\nc \NGo {\on {NG}}
\nc \Oo {\on O}
\nc \ch {\on {ch}}
\nc \rko {\on {rk}}
\nc \Sing {\on {Sing\ }}
\nc \Reg {\on {Reg}}
\nc \CoI {\on {CI}}
\nc \CoM {\on {CM}}
\nc \Gor {\on {Gor}}
\nc \Type {\on {Type}}
\nc \can {\on {can}}
\nc \Top {\on {T}}
\nc \rel {\on {rel}}
\nc \sgn {\on {sgn }}
\nc \trdeg {\on {tr.deg}}
\nc \codim {\on {codim }}
\nc \coht {\on {coht}}
\nc \divo {\on {div \ }}
\nc \coh {\on {coh}}
\nc \Clo {\on {Cl}}
\nc \embdim{\on {embdim}}
\nc \embcodim{\on {embcodim \ }}
\nc \qcoh {\on {qcoh}}
\nc \grad {\on {grad}\ }
\nc \grade {\on {grade}}
\nc \hto {\on {ht}}
\nc \depth {\on {depth}}
\nc \prof {\on {prof}}
\nc \reso{\on {res}}
\nc \ind{\on {ind}}
\nc \prodo{\on {prod}}
\nc \coind{\on {coind}}
\nc \Con{\on {Con}}
\nc \Crit{\on {Crit}}
\nc \Der{\on {Der}}
\nc \Char{\on {Char}}
\nc \Ch{\on {Ch}}

\nc \Ext{\on {Ext}}
\nc \Eo{\on {E}}
\nc \End{\on {End}}
\nc \Ad{\on {Ad}}
\nc \gr{\on {gr}}
\nc \Fo{\on {F}}
\nc \Gr{\on {Gr}}
\nc \Go{\on {G}}
\nc \GFo{\on {GF}}
\nc \Glo{\on {Gl}}
\nc \Ho{\on {H}}
\nc \CMo{\on {\CM}}
\nc \SCM{\on {SCM}}
\nc \hol{\on {hol}}
\nc{\sgd}{\on{sgd}}
\nc \supp{\on {supp}}
\nc \ssupp{\on {s-supp}}
\nc \singsupp{\on {singsupp}}
\nc \msupp{\on {msupp}}
\nc \spec{\on {spec}}
\nc \spano{\on {span }}
\nc \Span{\on {Span }}
\nc \Max{\on {Max}}
\nc \Min{\on {Min}}
\nc \Mod{\on {Mod}}
\nc \Rad {\on {Rad}}
\nc \rad {\on {rad}}
\nc \rank {\on {rank}}
\nc \range {\on {range}}
\nc \Slo{\on {SL}}
\nc \soc {\on {soc}}
\nc \Irr {\on {Irr}}
\nc \Imo {\on {Im}}
\nc \SSo{\on {SS}}
\nc \lub{\on {lub}}
\nc \gldim{\on {gl.d.}}
\nc \pdo{\on {p.d.}} 
\nc \ido{\on {i.d.}} 
\nc \dSSo{\dot {\SSo}}
\nc \So{\on S}
\nc \Io{\on I}
\nc \Jo{\on J}
\nc \jo{\on j}
\nc \Ko{\on K}
\nc \PBW{\Ac_{PBW}}
\nc \Ro{\on R}
\nc \To{\on T}
\nc \Ao{\on A}

\nc \Do{{\on D}}
\nc \Bo{\on B}
\nc \Po{\on P}
\nc \Qo{\on Q}
\nc \Zo{\on Z}
\nc \U{\on U}
\nc \wt{\on {wt}}
\nc \Uh{\hat {\U}}
\nc \T{\on T}
\nc \Lo{\on L}
\nc{\dop}{\on d}
\nc{\eo}{\on e}
\nc{\ado}{\on{ad}}
\nc{\Tot}{\on{Tot}}
\nc{\Aut}{\on{Aut}}
\nc{\sinc}{\on {sinc}}

\nc{\overrightleftarrows}[2]{\overset{#1}{\underset{#2}{\rightleftarrows}}}

\nc{\CCF}{\cal{CF}}

\nc{\CDF}{\cal{DF}}
\nc{\CHC}{\check{\cal C}}

\nc{\Cone}{\on{Cone}}
\nc{\dec}{\on{dec}}

\nc{\Diff}{\on{Diff}}

\nc{\dirlim}{\underset{\to}{\on{lim}}}
\nc{\dpar}{\partial}
\nc{\GL}{\on{GL}}
\nc{\CGr}{\cal{G}r}

\nc{\pr}{\on{pr}}
\nc{\semid}{|\!\!\!\times}
\nc{\Hom}{\on{Hom}}
\nc \RHom{\on {RHom}}

\nc \Proj{\mathrm {Proj\ }}
\nc \proj{\mathrm {proj}}
\nc{\Id}{\on{Id}}
\nc{\id}{\on{id}}
\nc{\Ima}{\on{Im}}
\nc{\invtimes}{\underset{\gets}{\otimes}}
\nc{\invlim}{\underset{\gets}{\on{lim}}}
\nc{\Lie}{\on{Lie}}

\nc{\re}{\on{Re }}
\nc{\Pic}{\on{Pic }}
\nc{\LPic}{\on{LPic }}
\nc{\Sch}{\on{Sch}}
\nc{\Sh}{\on{Sh}}
\nc{\Set}{\on{Set}}
\nc{\spo}{\on{sp\  }}
\nc{\Spec}{\on{Spec}}
\nc{\mSpec}{\on{mSpec}}
\nc{\Specb}{\bold {Spec}}
\nc{\Projb}{\bold {Proj}}
\nc{\Specan}{\on{Specan}}
\nc{\Spo}{\on{Sp}}
\nc{\Spf}{\on{Spf}}
\nc{\sym}{\on{sym}}
\nc{\symm}{\on{symm}}
\nc{\rop}{\on{r}}
\nc{\Td}{\on{Td}}
\nc{\Tor}{\on{Tor}}

\nc{\Artin}{\cal{A}rtin}
\nc{\Dgcoalg}{\cal{D}gcoalg}
\nc{\Dglie}{\cal{D}glie}
\nc{\Ens}{\cal{E}ns}
\nc{\Fsch}{\cal{F}sch}
\nc{\Groupoids}{\cal{G}roupoids}
\nc{\Holie}{\cal{H}olie}
\nc{\Mor}{\cal{M}or}

\nc{\CF}{\ensuremath{\cal{F}}}
\nc \Kc{\ensuremath{\cal K}}
\nc \Lc{{\ensuremath{\cal L}}}
\nc \lcc{{\mathcal l}} 
\nc \CC{{\ensuremath{\cal C}}} 
\nc \Cc{{\ensuremath {\cal C}}}
\nc \Pc{{\ensuremath{\cal P}}}
\nc \Dc{\ensuremath{\mathcal D}}
\nc \Ac{{\ensuremath{\cal A}}} 
\nc \Bc{{\ensuremath{\cal B}}}
\nc \Ec{{\ensuremath{\cal E}}}
\nc \Fc{{\ensuremath{\cal F}}}
\nc \Mcc{{\ensuremath{\cal M}}} 
\nc \hM{\hat{\Mcc}} 
\nc \bM{\bar {\Mcc}} 
\nc\hbM{\hat{\bar \Mcc}}  
\nc \Nc{{\ensuremath{\cal N}}}
\nc \Hc{{\ensuremath{\cal H}}} 
\nc \Ic{{\ensuremath{\cal I}}} 
\nc \Oc{\ensuremath{{\cal O}}}
\nc \Och{\hat{\cal O}} 
\nc \Sc{{\ensuremath{{\cal S}}}}
\nc \Tc{\ensuremath{{\cal T}}} 
\nc \Vc{{\ensuremath{{\cal V}}}} 
\nc{\CA}{{\ensuremath{{\cal A}}}}
\nc{\CB}{{\ensuremath{{\cal B}}}}
\nc{\C}{{\ensuremath{{\cal F}}}}
\nc{\Gc}{{\ensuremath{{\cal G}}}}
\nc{\CH}{\ensuremath{\mathcal H}}
\nc{\CI}{{\ensuremath{{\cal I}}}}
\nc{\CM}{{\ensuremath{{\cal M}}}}
\nc{\CN}{{\ensuremath{{\cal N}}}}
\nc{\CO}{{\ensuremath{{\cal O}}}}
\nc{\Rc}{{\ensuremath{{\cal R}}}}
\nc{\CT}{{\ensuremath{\mathcal T}}}
\nc{\CU}{\ensuremath{{\cal U}}}
\nc{\CV}{\ensuremath{{\cal V}}}
\nc{\CZ}{\ensuremath{{\cal Z}}}
\nc{\Homc}{\ensuremath{{\cal {Hom}}}}

\nc{\Tab}{\ensuremath{{\mbox{Tab}}}}

\nc{\STab}{\ensuremath{{\mbox{STab}}}}

\nc{\fa}{\frak{a}}
\nc{\fA}{\frak{A}}
\nc{\fg}{\frak{g}}
\nc{\fh}{\frak{h}}
\nc{\fI}{\frak{I}}
\nc{\fK}{\frak{K}}
\nc{\fm}{\frak{m}}
\nc{\fP}{\frak{P}}
\nc{\fS}{\frak{S}}
\nc{\ft}{\frak{t}}
\nc{\fX}{\frak{X}}
\nc{\fY}{\frak{Y}}

\nc{\bF}{\bar{F}}
\nc{\bCP}{\bar{\cal{P}}}
\nc{\bm}{\mbox{\bf{m}}}
\nc{\bT}{\mbox{\bf{T}}}
\nc{\bS}{\mbox{\bf{S}}}
\nc{\hB}{\hat{B}}
\nc{\hC}{\hat{C}}
\nc{\hP}{\hat{P}}
\nc{\htest}{\hat P}

\nc{\nen}{\newenvironment}
\nc{\ol}{\overline}
\nc{\ul}{\underline}
\nc{\ra}{\to}
\nc{\lla}{\longleftarrow}
\nc{\lra}{\longrightarrow}
\nc{\Lra}{\Longrightarrow}
\nc{\Lla}{\Longleftarrow}
\nc{\Llra}{\Longleftrightarrow}
\nc{\hra}{\hookrightarrow}
\nc{\iso}{\overset{\sim}{\lra}}

\nc{\dsize}{\displaystyle}
\nc{\sst}{\scriptstyle}
\nc{\tsize}{\textstyle}

\nen{exa}[1]{\label{#1}{\bf Example.\ } }{}

\nen{rem}[1]{\label{#1}{\em Remark.\ } }{}
\nen{exer}[1]{\label{#1}{\em Exercise.\ } }{}

\newcommand{\la}{\lambda}

\newcommand {\bC} {\mathbb C}

\newcommand{\bN}{\mathbb N}
\newcommand {\D} {\mathcal D}

\newcommand{\bla}{{\bf\lambda}}

\newcommand{\DD}{{\mathcal{D}}_V}

\newcommand{\OO}{{\mathcal{O}}_U}

\begin{document}
\title{Invariant differential operators and the generalized symmetric group}

\author{Ibrahim Nonkan\'e, Lat\'evi M. Lawson}

\address{Departement d'\'economie et de math\'ematiques appliqu\'ees, IUFIC, Universit\'e Thomas Sankara, Burkina faso}
\email{ibrahim.nonkane@uts.bf}

\address{ African Institute for Mathematical Sciences (AIMS), Summerhill Estates, East Legon Hills, Santoe, Accra,  Ghana P.O. Box LG DTD 20046, Legon, Accra, Ghana}
\email{latevi@aims.edu.gh}

\maketitle
\begin{abstract}
In this paper we study  the decomposition of the direct image of  $\pi_+(\Oc_{X})$ the polynomial ring  $\Oc_X$ as a $\D$-module, under the map $\pi: \spec \Oc_{X} \to \spec \Oc_{X}^{G(r,n)}$, where $\Oc_{X}^{G(r,n)}$  is the ring of invariant polynomial under the action of the wreath product $G(r,p):= \ZZ / r \ZZ \wr \Sc_n $. We first describe the generators of the simple components of $\pi_+(\Oc_X)$ and give their multiplicities. Using an equivalence of categories and the higher Specht polynomials, we describe a $\D$-module decomposition of the  of the polynomial ring localized  at the discriminant of $\pi$. Furthermore, we study the action invariants differential operators on the higher Specht polynomials.\\\\
 
\textbf{keywords}:\ { Direct image, Differential structure,  complex reflection groups, Representation theory, wreath product, Higher Specht polynomials,  Primitive idempotents, Reflection groups, Young diagram.}
\end{abstract}


\section{Introduction} 
A fundamental problem in representation theory is the description of all irreducible representations. We are interested in the polynomial representation the invariant differential operators with respect to the generalized symmetric group.
We know that the direct image of a simple module for a proper map in semi-simple by the Decomposition Theorem \cite{Cataldo}. The simplest case is when the map $\pi: X= \spec B \to  Y= \spec A$ is finite, in which case it is easy to give an elementary and wholly algebraic proof, using essentially the (generic) correspondence with the differential Galois group, which equals the ordinary Galois group $G$. The irreducible submodule of the direct image are in one-to-one correspondence with the irreducible representations of $G$ (see \cite{IBK}). In this paper  we make the differential structure more explicit in the case of the invariants of the complex reflection group $G(r,n),\  B=\bC[x_1,\ldots, x_n] \subset A =\bC[x_1,\ldots,x_n]^{G(r,n)}.$
We explicitly study the simple component of the direct image $\pi_+(\Oc_X)$ of the polynomial ring $\Oc_X$ as a $\Dc$-module  under the map $  \pi: \spec \Oc_X \to \spec \Oc_Y$ where $ \Oc_Y= \Oc_X^{G(r,n)};$ the ring of invariant polynomials under the action of ${G(r,n)}.$  
We describe the generators of the simple components of $\pi_+(\Oc_X)$ and their multiplicities as in \cite{IBK}.
 We thus establish the decomposition structure   of $\pi_+(\Oc_X)$ by  means of the higher  Specht polynomials.  This proof uses the fact  that the irreducible $\D$-submodules of $\pi_+(\Oc_X)$ are in one-to-one correspondence with irreducible representations of  $G(r,n) $. \\
 Secondly we  elaborate a $\D$-module decomposition of the polynomial ring localized  at the discriminant of $\pi.$\\
 Finally we describe the action of invariant differential operators on  higher Specht polynomials.
The higher Specht polynomials (introduced combinatorially  \cite{Ariki}), are adapted to the $\D$-module structure.\\
This  paper  generalizes results on modules over the Weyl algebra appeared in \cite{IBK} and \cite{RTG}.
The case $r=2$ have been presented at 10th International Conference on Mathematical Modeling in Physical Sciences in order to describe the action of the rational Olshanetsky-Perelomov operator Hamiltonian on polynomials  \cite{IB}.

\section{Preliminaries}

\subsection{Direct image} 
We briefly recall the definition of the direct image of a $\D$-module \cite{CSC}.\\
Let $K$ be a field of characteristic zero,  put $X=K^n$. The polynomial ring $K[x_1,\ldots, x_n]$ will be denoted by $K[X];$ and the Weyl algebra generated by $x_i$'s and $\frac{\partial}{\partial x_i}$'s by $\D_X$. The $n$-tuple $(x_1, \ldots,x_n)$ will be denoted by $X$. Similar conventions will holds for $Y=K^m$, with polynomial ring $K[Y]$ and Weyl algebra $\D_Y$.\\

 Let $\pi  : X \to Y$ be a polynomial map, with $\pi=(\pi_1,\ldots,\pi_m)$. Let $M$ be a left $\D_Y$-module. The inverse image of $M$ under the map  $\pi$ is $\pi^+( M)= K[X] \otimes_{K[Y]}M.$ This is a $K[X]$-module. It becomes a $\D_X$-module with $\partial_{x_i}$ acting according to the formula
$$\frac{ \partial }{\partial x_i} (h \otimes u) = \frac{\partial h}{\partial{x_i}}\otimes u + \sum_{j=1}^m \frac{\partial  \pi_j}{\partial x_i} \otimes \frac{\partial }{\partial y_j} u, \ h \in K[X], u \in M.$$
Since $\D_Y \otimes_{\D_Y} M \cong M,$ we have that 
$$ \pi^+ (M) \cong K[X] \otimes_{K[Y]} \D_{Y} \otimes_{\D_Y}M = \pi^+(K[Y])  \otimes_{\D_Y}M.$$
Writing $D_{X \to Y}$ for $\pi^+( K[Y])$, on has that $\pi^+(M)= \D_{X \to Y}  \otimes_{\D_Y}M.$
Note that $\D_{X\to Y}$ is $\D_X$- $\D_Y$-bimodule.

Let $N$ be a right $D_X$-module. The tensor product 
$$\pi_+(N) =N \otimes_{\D_X} \D_{X \to Y}$$ is a right $\Dc_Y$-module, which is called the {\it  direct image} of $N$ under the polynomial map $\pi$. 
Let us consider the {\it standard transposition} $ \tau : \D_X \to \D_X$ defined by $ \tau (h \partial ^{\alpha} )= (-1)^{|\alpha|} \partial^{\alpha} h,$ where $h \in K[X]$ and $\alpha \in \bN^n.$ If $N$ is a right $\D_X$-module then we define a left $\D_X$-module $N^t$ as follows. As an abelian group, $N^t=N.$ If $a \in \D_X$ and $ u \in N^t$ then the left action of $a$ on $u$ is defined by $a \star u= u \tau(a).$ Using the standard transposition for $\D_Y$ and $\D_X$, put $D_{ Y\leftarrow X} = (D_{ X \to Y} )^t$, this is a $\D_Y$-$\D_X$-bimodule. Let $M$ be a left $\D_X$-module. The direct image of $M$ under $\pi$ is defined by the formula
$$ \pi_+(M)= D_{ Y \leftarrow X} \otimes_{\D_X} M.$$ It is clear that $\pi_+(M)$ is  a $\D_Y$-module.\\
The following is the Kashiwara decomposition theorem
\begin{theorem} \cite{Cataldo} 
Let $\pi: X\to Y$ be a polynomial map. If $M$ is a a simple (holonomic) module over $\D_X$. Then $\pi_+ (M)$ is a semisimple $\D_Y$-module. we have 
$$ \pi_+(M)= \oplus  M_i^{\alpha_i},$$ where the $M_i$ are inequivalent irreducible $\D_Y$-submodules.
\end{theorem}

\subsection{Higher Specht Polynomials for Reflections group} In this subsection we recall  some  facts about the representation of the generalized symmetric group $G(r,n)$ \cite{Ariki}. 
Let $\Sc_n$ be the group of permutations of the set  $\{1, \ldots,n\}$ and $\ZZ /r \ZZ$ be the cyclic group of order $r$.\\
The  generalized symmetric group $G(r,n)$ is  the semi-direct product of $(\ZZ/r \ZZ)^n$ with $\Sc_n$, written as $(\ZZ / r\ZZ)^n \rtimes \Sc_n$, where $(\ZZ / r\ZZ)^n $ is the direct product of $n$ copies of $\ZZ / r\ZZ $. Let $\xi$ be a primitive $r$-th root of  1.
$(\ZZ / r\ZZ)^n \rtimes \Sc_n= \{ (\xi^{i_1}, \ldots, \xi^{i_n}; \sigma) |\  i_k\in \bN,\  \sigma \in \Sc_n \}$, whose product is given by 
$$  (\xi^{i_1}, \ldots, \xi^{i_n}; \sigma) (\xi^{j_1}, \ldots, \xi^{j_n}; \pi)=  (\xi^{i_1 + j_{\sigma^{-1} (1)}}, \ldots, \xi^{i_n + j_{\sigma^{-1} (n)}},; \sigma \pi).$$
Let $\Oc_X= \bC[x_1, \ldots,x_n]$ be the ring of polynomials in $n$ indeterminates  on which the group $G(r,n)$ acts as follows:

$$ (\xi^{i_1}, \ldots, \xi^{i_n}; \sigma) f= f(\xi^{i_{\sigma(1)}} x_{\sigma(1)}, \ldots, \xi^{i_{\sigma(n)}}x_{\sigma(n)}; \sigma),$$
where $f \in \Oc_X$ and  $(\xi^{i_1}, \ldots, \xi^{i_n}; \sigma) \in G(r,n)$. It is known that the fundamental invariants under this action are given by the elementary symmetric functions $e_j(x_1^r, \ldots, x_n^r),\ 1 \leq j \leq n.$ Let $J_{+}$ be the ideal of $\Oc_X$ generated by these fundamental invariants and $\Lambda=\Oc_X/ J_{+}$ be the quotient ring. It is also known that the $G(r,n)$-module $\Lambda $ is isomorphic to the group ring $\bC[G(r,n)]$, namely the left regular representation. A description of all irreducible components of $\Lambda $ is known in \cite{Ariki}, in terms of what is called  "higher Specht polynomials". The irreducible representation of $G(r,n)$ are parametrized by the $r$-tuple of Young diagrams $(\lambda^1,\ldots, \lambda^r)$ with $|\lambda^1| +\cdots + |\lambda^r|=n.$

 Let $\mathcal{P}_{r,n}$ be the set of $r$-tuples of Young diagrams $\bla=(\lambda^1, \ldots, \lambda^r)$ with $|\lambda^1| +\cdots + |\lambda^r|=n$. By filling each cell with a positive integer in such a way that every  $j\ ( 1 \leq j \leq n)$ occurs once, we obtain an $r$-tableau $T=(T^1,\ldots,T^r)$ of shape ${\bf \bla}= (\lambda^1, \ldots, \lambda^r)$. When the number $k$ occurs  in the component $T^i$, we write $k \in T^i$. The set of $r$-tableaux of shape $\bla$ is denoted by $\Tab(\bla).$ An $r$-tableau $T=(T^1,\ldots,T^r)$ is said to be standard if the numbers are increasing on each column and each row of $T^\nu\ (1 \leq \nu \leq  r)$. The set of $r$-standard tableaux of shape $\bla$ is denoted by $\STab(\bla).$

Let $S =(S^1, \ldots, S^r) \in \STab (\bla)$. We associate a word $w(S)$ in the following way. First we read each column of the component $S^1$ from the bottom to the top starting from the left. We continue this procedure for the tableau $S^2$ and so on. For word $w(S)$ we define index $i(w(S))$ inductively as follows. The number 1 in the word $w(S)$ has the index $i(1)=0.$ If the number $k$ has index $i(k)=p$ and the number has number $k+1$ is sitting on the left (resp. right) of $k,$ then $k+1$ has index $p+1$ (resp. $p$). Finally, assigning the indices to the corresponding cells, we get a shape ${\bf \bla}= (\lambda^1, \ldots, \lambda^r)$, each cell filled with a nonnegative integer, which is denoted by $i(S) =(i(S)^1, \ldots, i(S)^r).$

Let $T=(T^1,\ldots, T^r)$ be an $r$-tableau of shape $\bla$. For each component $T^\nu\ (1\leq \nu \leq r),$ the Young symmetrizer $\eb_{T^\nu}$ of $T^\nu$ is defined by
$$ \eb_{T^\nu}= \frac{1}{ \alpha_{T^\nu}} \sum_{\sigma \in R(T^\nu)\ \tau \in C(T^\nu)} \sgn (\tau) \tau \sigma,$$ where $\alpha_{T^\nu}$ is the product of the hook lengths for the shape $\lambda^\nu$,  $R(T^\nu)$ and $C(T^\nu)$ are the \textit{row-stabilizer} and \textit{colomn-stabilizer} of $T^\nu$ respectively.\\
We may regard a tableau $T$ on a Young diagram $\lambda$ as a map
$$ T: \{ \mbox{cells\ of}\  \lambda \}  \to \ZZ_{\geq 0},$$
which assigns to a cell $\xi$ of $\lambda$ the number $T(\xi)$ written in the cell $\xi$ in $T$.\\
For $S\in \STab(\bla)$ and $T \in\Tab(\bla)$, Ariki, Terasoma and Yamada in \cite{Ariki} defined the higher Specht polynomial for $G(r,n)$ by
$$ F_{T}^{S}= \prod_{\nu =1}^r \bigg( \eb_{T^\nu} (x_{T^\nu}^{ri(S)^{\nu}}) \prod_{k\in T^\nu} x_k^\nu \bigg),$$ where $$\displaystyle x_{T^\nu}^{ri(S)^{\nu}}= \prod_{\xi \in \lambda^\nu} x_{T^\nu(\xi)}^{ri(S)^\nu(\xi)}.$$

The following is the fundamental result in \cite{Ariki} on the higher Specht polynomials for $G(r,n).$
\begin{theorem}
\begin{enumerate}
\item The space $\displaystyle V_S(\lambda) =\sum_{T\in \Tab(\lambda)} \bC F_T^S$ affords an irreducible representation of the reflection group $G(r,n)$.
\item The set $\{ F_T^S \ | \ T \in \STab(\bla) \}$  gives a basis over $\bC$ for $V_S(\bla).$ 
\item For $S_1 \in \STab(\bla)$ and $S_2 \in \STab(\mu)$, the representation $V_{S_1}(\bla)$ and $V_{S_2}(\mu)$ are isomorphic if and only if $S_1$ and $S_2$ has the same shape, i.e. $\lambda=\mu.$

\item We have the irreducible decomposition 
$$ \bC[G(r,n)]= \bigoplus_{\bla \in \mathcal{P}_{r,n}} \bigoplus_{S\in \STab(\bla)} V_S(\lambda)$$
as representation of $G(r,n)$.
\end{enumerate}
\end{theorem}

\begin{theorem}

The higher Specht polynomials in $\C=\{ F_T^S; S, T \in \STab(\lambda), \lambda \vdash n \}$
form a basis of the $\bC[x_1,...,x_n]^{G(r,n)}$-module $\bC[x_1,...,x_n]$
\end{theorem}

\section{Decomposition Theorem} 

We are interested in studying the decomposition structure of $\pi_+(M)$, where $M=\Oc_X,\ \pi :X= \spec(\Oc_X)\to Y=\spec(\Oc_X^{G(r,n)})$. Since $\Oc_X$ is a holonomic $\D_X$-module \cite[Chapter 10]{CSC}, $\pi_+(\Oc_X)$ is a semisimple $\D_Y$-module by the Kashiwara decomposition theorem. We  construct the simple components of  $\pi_+(\Oc_X)$ and provide their multiplicities.  Let us recall some useful facts from \cite{IBK}.\\

Let $\Delta:= Jac (\pi)$ be the Jacobian of $\pi$, $\Delta^2$ the discriminant of $\pi$ we denote the complement of the branch locus and the discriminant by $U$ and $V$, respectively.
Assume now that $U,V$ are such that the respective canonical modules are generated by volume forms $dx$, and $dy$,  related by $dx=\Delta dy$, where  $\Delta$ is the Jacobian of $\pi$. 
 \begin{proposition}  
\label{prop:genericstudy}
\begin{itemize}
 \item[(i)] There is an isomorphism of $\D_V$-modules $$T:\pi_+(O_U)\cong O_U,\quad r(  dy^{-1}\otimes dx)\mapsto r \Delta ^{-1}.$$
\item[(ii)]
  $T(\pi_+(O_X))$ is isomorphic as a $\D_Y$-module to $\pi_+(\Oc_X)$.
\end{itemize}
\end{proposition}
\begin{proof}
See \cite[Lemma 2.3]{IBK}.
\end{proof}
It is more convenient to study $T(\pi_+(O_X))\cong \pi_+(O_X)$, as a submodule of $O_U$, than using the definition of $\pi_+(O_X)$. Therefore to reach our goal, we will first  study the decomposition of $O_U$ into irreducible components as a $\D_V$-module.\\
The following proposition enables us to reduce the study of the decomposition factors of $\pi_+(O_X)$ 
to the behavior of the direct image over the complement to the branch locus, or even over the generic point. Let $j:U\hookrightarrow X$  and $i:V\hookrightarrow Y$ be the inclusions.

\begin{proposition}  \label{prop:etale.semisimple2}
Let $\pi: X\to Y$ be a finite map. Then
\begin{itemize}
\item[(i)]
 $\pi_+(O_X)$ is semi-simple as a $\D_Y$-module.
\item[(ii)] If $\pi_+(O_X)=\oplus M_k,\ k\in I$ is a decomposition into simple
(non-zero) $D_Y$-modules, then $\pi_+(O_U)=\oplus i^+(M_k),\ k\in I$, is a decomposition of  $\pi_+(O_U)$ into simple (non-zero) $\Dc_{V}$-modules.
\end{itemize}
\end{proposition} 
\begin{proof}
See \cite[Proposition 2.8]{IBK}.
\end{proof}
\subsection{Action description}
As we want to study the polynomial representation of a ring of invariant differential operators localized at $\Delta^2,$ it is convenient to precisely describe the action of that ring  on the polynomials ring.\\

Let $\Dc_X=\bC \langle x_1, \ldots, x_n, \frac{\partial}{\partial x_1}, \ldots, \frac{\partial}{\partial x_n} \rangle$ be the ring of differential operators associated with the polynomial ring $\Oc_X=\bC [x_1,\ldots,x_n]$, and $\Oc_Y=\bC[x_1, \ldots,x_n]^{G(r,n)}= \bC[y_1,\ldots,y_n]$ be the ring of invariant polynomials under the real reflection group $G(r,n).$ 

We denote by $\Dc_Y= \bC \langle y_1, \ldots, y_n, \frac{\partial}{\partial y_1}, \ldots, \frac{\partial}{\partial y_n} \rangle $ the ring of differential operators associated with  $\Oc_Y= \bC [y_1,\ldots,y_n]$.  By \cite{LS1}, $\Dc_Y$ is the ring of invariant  differential operators under the action of the  reflection group $G(r,n)$.

\textbf{Notations} We adopt the following notations  $$\displaystyle \OO:=\bC [x_1,\ldots,x_n, \Delta^{-1}],\  \Oc_{V} := \bC [y_1,\ldots,y_n, \Delta^{-2}],\  \DD:= \bC \langle y_1, \ldots, y_n, \frac{\partial}{\partial y_1}, \ldots, \frac{\partial}{\partial y_n}, \Delta^{-2} \rangle.$$

\begin{lemma}

$\OO$ is a $\DD$-module. 
\end{lemma}
\begin{proof} 
Let us make clear the  action of $\DD$ on $\OO.$\\
Let $\displaystyle A=  \displaystyle \bigg(\frac{\partial y_j}{ \partial x_i} \bigg)_{1\leq i,j \leq n}$ so that $\det(A)= \Delta$. We get the following equation
$$ \left (\begin{array}{ccc} \frac{\partial}{\partial x_n}\\  \vdots\\ \frac{\partial}{\partial x_n} \end{array} \right)= A  \left (\begin{array}{ccc} \frac{\partial}{\partial y_1} \\ \vdots \\  \frac{\partial }{\partial y_n} \end{array} \right).$$
 It follows that 
$$ 
\left (\begin{array}{ccc} \frac{\partial}{\partial y_1} \\ \vdots \\  \frac{\partial }{\partial y_n} \end{array} \right)= A^{-1} \left (\begin{array}{ccc} \frac{\partial}{\partial x_n}\\  \vdots\\ \frac{\partial}{\partial x_n} \end{array} \right)
$$
and it is now clear that $\OO$ is a $\DD$-module. 
\end{proof}

We know that  $\OO$ a $\DD$-  semisimple module. what are the simple components of $\OO$ as $\DD$-module  and their multiplicities?\\

\subsection{Simple components and their multiplicities}

In this section, we state our main result. We use the representation theory of the generalized symmetric group $G(r,n)$ to yield results on modules over the ring of differential operators. It is well-known that 
$$ \Oc_X= \bC[G(r,n)] \otimes \Oc_Y \  \mbox{as } \  \Oc_Y\mbox{-modules}.$$ 
Let us consider the multiplicative closed set $S=\{ \Delta^{k} \}_{k \in \bN} \subset \Oc_X.$ It follows that:
$$ S^{-1} \Oc_X= \bC[G(r,n)] \otimes S^{-1}\Oc_Y \  \mbox{as } \  S^{-1}\Oc_Y\mbox{-modules}.$$ where $S^{-1} \Oc_X$ and $S^{-1}\Oc_Y$ are the localizations of  $\Oc_X$ and $\Oc_Y$ at $S$ respectively. But $S^{-1} \Oc_X= \OO$ and $S^{-1} \Oc_Y= \Oc_{V}$ , whereby we get
$$ \OO=  \bC[G(r,n)] \otimes \Oc_{V} \  \mbox{as} \  \bC[G(r,n)]\mbox{-modules}.$$

\begin{lemma}
There exists an injective map
$$ \bC[G(r,n)] \hookrightarrow  \Homo_{\bC}( \OO, \OO ).$$
\end{lemma}
\begin{proof}
The $\bC[G(r,n)]$-module  $\bC[G(r,n)]$ acts faithfully on itself by multiplication, and this multiplication yields an injective map $\bC[G(r,n)] \hookrightarrow \Homo_{\bC} \bigg( \bC[G(r,n)] , \bC[G(r,n)]\bigg).$ Since $\Oc_{V} $ is invariant under this action of $\bC[G(r,n)]$, we get the expected injective map. 
\end{proof}

\begin{proposition}
There exists an injective map 
$$ \bC[G(r,n)] \hookrightarrow  \Homo_{\Dc_{V}}( \OO , \OO ).$$
\end{proposition}
\begin{proof}
Since  $\Dc_{V}= \bC \langle y_1, \ldots, y_n, \partial y_1, \ldots, \partial y_n, \Delta^{-2}  \rangle $, we only need to show that every element of $\bC[G(r,n)]$ commutes with $y_1,\ldots,y_n, \partial y_1, \ldots, \partial y_n.$ 
\begin{enumerate}
\item[$\bullet$] It is clear that every element of $\bC[G(r,n)]$ commutes with $y_i,\ i=1, \ldots,n$.
\item[$\bullet$] Let us show that every element of $\bC[G(r,n)]$ commutes with $\partial y_i,\ i=1, \ldots,n$.
Let ${\bf D}$ be a derivation on the field $ K=\bC(y_1, \ldots, y_n)$  of fractions of $\Oc_V$, then $(K, {\bf D})$ is a differential field. Let $L=\bC(x_1, \ldots, x_n)$ be the field of fractions of $\OO$. We have that  $K=L^{G(r,n)}$ is the fixed field and $L$ is a Galois extension of $K$, with Galois group $G(r,n)$.  Then by \cite[Th\'eor\`eme 6.2.6]{Cham}  there exists a unique derivation on $L$ which extends ${\bf D},$ then  $(L , {\bf D})$ is also a differential ring. In this  way, $\sigma^{-1}  {\bf D} \sigma = {\bf D}$ for every $\sigma \in G(r,n).$ Therefore $\sigma {\bf D}= {\bf D}\sigma$ and $ \sigma $ commute with ${\bf D}.$   
\end{enumerate}
\end{proof}

\begin{corollary}
$$\bC[ G(r,n)] \cong \Homo_{\DD} ( \OO , \OO )$$ 
\end{corollary}
\begin{proof}
see \cite[Corollary 26 ]{IBK}
\end{proof}

Before we state our  main result, let us recall some facts.\\
By Maschke's Theorem \cite[Chap XVIII]{Lan}, we know that $\bC[G(r,n)]$ is a semi-simple ring, and  

$$\bC[G(r,n)] = \displaystyle \bigoplus_{ \lambda \in \Pc_{r,n}} R_{\lambda}, $$

 where  $\Pc_{r,n}$ be the set of $r$-tuples of Young diagrams   and  $R_{\lambda}$ are simple rings. In fact $ \displaystyle R_\bla = \bigoplus_{ S \in \STab(\bla) } V_S (\bla) $ ( see  Theorem 2.2).  
  We have the following corresponding decomposition of the identity element of $\bC[G(r,n)]$:
$$ 1=\sum_{ \lambda \in \Pc_{r,n}}  r_{\lambda},$$
 where $r_{\lambda}$ is the identity element of $R_{\lambda}$, with  $r_{\lambda} ^2=1$ and $r_{\lambda} r_{\mu} =0$ if $\lambda \neq \mu.$  $\{ r_{\lambda} \}_{ \lambda \in \Pc_{r,n}}$ is the set of primitive central idempotents of $\bC[G(r,n)]$. In fact  $ R_{\lambda} = r_{\lambda} \bC[ G(r,n)]$, for $\lambda \in \Pc_{r,n} $.
Let $n \in \bN^*, \bla \in  \Pc_{r,n}$, we set  $ \displaystyle \Tab(n)=\displaystyle \cup_{\lambda \in \Pc_{r,n}} \Tab(\lambda)$ and $\displaystyle \STab(n)=\cup_{\lambda \in \Pc_{r,n}} \STab(\lambda)$.

\begin{theorem}

For every  primitive idempotent $e \in \bC[G(r,n)]$.
\begin{enumerate}
\item $ e \OO$ is a nontrivial  $\Dc_{V}$-submodule of $\OO,$
\item The $\DD$-module $e \OO$ is simple,
\item There exist  $\lambda  \in \Pc_{r,n}$ and a higher Specht polynomial $F_T^S$  (with $S,T \in \STab(\lambda))$ such that $e  \OO= \DD F_T^S.$
\end{enumerate}
\end{theorem}
\begin{proof}\
\begin{enumerate}
\item Let $e  \in \bC[G(r,n)]$ be a  primitive  idempotent, we know that $\bC[G(r,n)]e $ is a $G(r,n)$- irreducible representation.  Theorem 2.3 states that  there is $\lambda \in \Pc_{r,n}$ and $S\in \STab(\lambda)$ such  that $\bC[G(r,n)] e \cong V_S(\lambda)$, with $V_S(\lambda) \subset \OO.$ By  \cite[ Chap III, \S 4, Theorem 3.9 ]{Boerner}, we have $ e \bC[G(r,n)]e \cong \bC e  \neq \{0\}$. $ \{0\} \neq e V_S(\lambda) \subset e \OO.$  Since $e $ commutes with every element of $\DD$ and $\OO$ is a $\DD$-module, it follows that $e  \OO$ is a nontrivial $\DD$-module.\\
In fact $V_S(\lambda)$ is a cyclic $\bC[G(r,n)]$-module, i.e.,  there exist  $T, S \in \STab(\lambda)$ and a higher Specht polynomial $F_T^S$ such that  $ V_S(\lambda) = \bC[G(r,n)]  F_T^S$, so that $ \bC[G(r,n)] e \cong \bC[G(r,n)] F_T^S$. Then it follows that there is a higher Specht polynomial $F_T^S$ such  that $e  F_T^S$ is a scalar multiple of $F_T^S.$
 
\item Assume that $1 =\sum_{i=1}^s e_i$ where the $\{ e_i \}_{ 1 \leq i \leq s}$ is the set of primitive idempotents of $\bC[G(r,n)] ,$ then $\OO= \sum_{i=1}^s e_i \OO.$ Let $m \in e_i \OO \cap e_j\OO$  with $ i \neq j$ so that $m=e_i m_i$ and $ m= e_j m,$ but $ e_i e_j=0$ then $ e_im=e_ie_jm=0$ hence $m=0.$ Therefore 
$ \OO = \oplus_{i=1}^s e_i \OO$  {and}  we get:   
$$ \Homo_{\DD} (\OO, \OO) \cong  \bigoplus_{i,j =1}^s   \Homo_{\DD} (e_i\OO, e_j\OO), $$
by Corollary 3.6   we know that
$ \displaystyle \bC[G(r,n)] \cong \bigoplus_{i,j =1}^s   \Homo_{\DD} (e_i\OO, e_j\OO).$
For every $\lambda \in \Pc_{r,n}$, we pick a unique  irreducible representation $V_S(\lambda)$ for a certain $S\in \STab(\lambda)$ which we denote by $V(\lambda):= V_S(\lambda)$.
We also have, by \cite[Proposition 3.29]{Fulton-Harris}, that  
$ \displaystyle \bC[G(r,n)] \cong \bigoplus_{\lambda \in \Pc_{r,n} } \mbox{End}_{\bC} (V({\lambda})).$  
But by the Wedderburn's decomposition Theorem \cite[Chap II,\S 4, Theorem 4.2]{Boerner} we also know that 
$$ \displaystyle \bC[G(r,n)] = \bigoplus_{\lambda \in \Pc_{r,n}} r_{\lambda} \bC[G(r,n)]\ \mbox{and}\ r_{\lambda} \bC[G(r,n)]  \cong \mbox{Mat}_{f^{\lambda}}(\bC) \cong \mbox{End}_{\bC} ( \bC^{f^{\lambda}}) $$ {where}  $f^{\lambda} = \dim_{\bC} (V(\bla)).$ Each $r$ standard tableau $T$ corresponds a primitive idempotent  $e_T$, so that $r_\la = \displaystyle  \sum_{T \in \STab(\la)} e_T.$
In what follows we denote $e_i$ the primitive idempotent associated  with standard tableau $T_i,$ ( i.e. $e_i=e_{T_i}$.) Let us show that  $$ \displaystyle \bC[G(r,n)] \cong  \bigoplus_{\lambda \vdash n} \bigg( \bigoplus_{T_i, T_j \in \STab(\lambda)}   \Homo_{\DD} (e_i\OO, e_j\OO) \bigg)\  \mbox{where}\ e_i=e_{T_i}.$$ \\
 Let $x$ be an element of $ \bC[G(r,n)]$ and $r_{\lambda}$ the primitive central idempotent associated with $\lambda \in \Pc_{r,n}.$ Then $x$ induces an $\DD$-homomorphism $\OO \to \OO, m \mapsto x\cdot m;$ the multiplication by $x$. Since $r_{\lambda}$ is in the centre of $\bC[G(r,n)] $,  $ x \cdot  (r_{\lambda} \OO )=(x \cdot r_{\lambda} ) \OO \subset r_{\lambda} \OO,$ which means $x \in \displaystyle \bigoplus_{\lambda \in \Pc_{r, n}} \Homo_{\DD} ( r_{\lambda} \OO, r_{\lambda} \OO).$  It follows that
 $$ \displaystyle \Homo_{\DD} (  \OO, \OO)  \cong  \bigoplus_{\lambda \in \Pc_{r,n}}    \Homo_{\DD} (r_\la \OO, r_\la\OO) .$$
  Then $\Homo_{\DD} ( e_i \OO , e_j \OO) =\{0 \}$ if $T_i \in \STab(\lambda_i), T_j \in \STab(\lambda_j)$ and  $\lambda_i \neq \lambda_j.$ 
  We  get that   $$ \displaystyle \Homo_{\DD} (  \OO, \OO)  \cong  \bigoplus_{\lambda \in \Pc_{r,n}} \bigg( \bigoplus_{T_i, T_j \in \STab(\lambda)}   \Homo_{\DD} (e_i\OO, e_j\OO) \bigg).$$ 
 The number of direct factors in the sum $\displaystyle  \bigoplus_{T_i, T_j \in \STab(\lambda)}   \Homo_{\DD} (e_i\OO, e_j\OO) $ is $(f^{\lambda})^2.$\\
Let us show that $  \Homo_{\DD} (e_i\OO, e_j\OO) \cong \bC$ if $T_i, T_j \in \Tab(\lambda).$  Consider the following commutative diagram:

\begin{center}
\begin{tikzpicture}
  \matrix (m) [matrix of math nodes,row sep=3em,column sep=4em,minimum width=2em]
  {
    \bC[G(r,n)] & \displaystyle \Homo_{\DD} (  \OO, \OO)  \\
     r_{\lambda} \bC[G(r,n)]& \displaystyle \Homo_{\DD} (  r_{\lambda}\OO, r_{\lambda}\OO) \\};
  \path[-stealth]
    (m-1-1) edge node [left] {$\alpha_{\lambda}$} (m-2-1)
            edge  node [below] {$\phi$} (m-1-2)
    (m-2-1.east|-m-2-2) edge node [below] 
    {$\psi$}
            node [above]{ }(m-2-2)
    (m-1-2) edge node [right] {$\beta_{\lambda}$} 
    (m-2-2); edge [dashed,-] (m-2-1);
\end{tikzpicture}
\end{center}

where
 $\displaystyle \beta_{\lambda}: \bigoplus_{\mu \in \Pc_{r,n}} \Homo_{\DD} ( r_{\mu} \OO, r_{\mu} \OO) \to 
 \Homo_{\DD} ( r_{\lambda} \OO, r_{\lambda}\OO)$ and \\ $ \displaystyle \alpha_{\lambda}: \bigoplus_{\mu \in \Pc_{r, n}} r_{\mu} \bC[G(r,n)] \to r_{\lambda} \bC[G(r,n)]$ are canonical projections et $\phi$ is the isomorphism in Corollary3.6. It follows that $\psi$ is an isomorphism hence $ r_{\lambda} \bC[G(r,n)] \cong   \displaystyle \Homo_{\DD} (  r_{\lambda}\OO, r_{\lambda}\OO).$ \\
 Now  we  identify $ r_{\lambda} \bC[G(r,n)]$ with either the set  $\mbox{Mat}_{f^{\lambda}}(\bC)$ of square matrices of order $f^{\lambda}$ with coefficients in $\bC$ either with  $\displaystyle \mbox{End}_{\bC}(\bC^{f^{\lambda}}).$
 Let $E_{ij}$ be the square matrix of order $f^{\lambda}$  with 1 at the position $(i,j)$ and 0 elsewhere and $E_i=E_{i,i}, $ then we identify the primitive  idempotent $e_i \in  r_{\lambda} \bC[G(r,n)]$ with $E_i$ in $\mbox{Mat}_{f^{\lambda}} (\bC).$ Let $B=(a_{ij}) \in \mbox{Mat}_{f^{\lambda}}(\bC)$ we get $B=\sum_{i,j} a_{i,j} E_{i,j}= \sum_{i,j} E_i B E_{j}$, in fact $E_iBE_j$ is  the matrix with $a_{i,j}$ in the position $(i,j)$ and 0 elsewhere, if $R= \mbox{Mat}_{f^{\lambda}}(\bC)$  we get that  $ E_i R E_j \cong \bC.$ 
 
  This isomorphism $\psi$ implies that $$ \displaystyle  \bigoplus_{T_i,T_j \in \STab(\lambda)} E_i R E_j \cong  \bigoplus_{ T_i,T_j \in \STab(\lambda)} \Homo_{\DD} (e_i\OO, e_j \OO);$$ the restriction of $ \psi$ to $E_iRE_j$ yields a map $E_i R E_j \to \Homo_{\DD}(e_i \OO, e_j \OO)$ and this map is surjective, moreover we have $E_i R E_j \cong \Homo_{\DD}(e_i \OO, e_j \OO)$. Therefore $ \Homo_{\DD}(e_i \OO, e_i \OO) \cong \bC.$ Let us assume that $e_i \OO$ is not  simple $\DD$-module, then $e_i\OO$ may be written as $ e_i \OO= \oplus_{j \in J} N_j$ where the $N_j$ are simple $\DD$-modules and $|J| > 1$. It follows that $ \dim_{\bC} (  \Homo_{\DD}(e_i \OO, e_i \OO))  \geq |J|$  but $ \Homo_{\DD}(e_i \OO, e_i\OO) \cong \bC$ so we obtain that $J=1$, which necessary implies that $e_i \OO$ is a simple $\DD$-module. 
 \item By the the proof (i) there exists a higher Specht polynomial $F_T^S  $, with $ S, T \in \STab(\bla)$ where $ \lambda \vdash n$ such that $ e_i \OO = \DD F_{T}^S.$
\end{enumerate}
\end{proof}

\begin{corollary}

With the above notations, $e_i\OO \cong_{\DD} e_j \OO$ if only if $T_i$ and $T_j$  have the same shape i.e. if there is a partition $\lambda \in \Pc_{r,n} $ such that $ T_i, T_j \in \STab(\lambda ).$
\end{corollary}

\begin{proof}
The  $\DD$-modules $e_i \OO$ are simple and $\Homo_{\DD}(e_i \OO, e_j \OO) \cong \bC$ whenever there exists a partition $\lambda \in \Pc_{r,n}$ such that $ T_i, T_j \in \STab(\lambda).$ Since $  \Homo_{\DD}(e_i \OO, e_j \OO)\neq \{0\},$ we conclude by using the Schur lemma.

\end{proof}

\begin{proposition}
  Let  $\lambda \in \Pc_{r,n}$, $ T \in \STab(\lambda) $, and let $e$ be the  primitive idempotent associated  with $T,$  denote by $F_T:=F_T^S$  the corresponding higher Specht polynomial (for some $S\in \STab(\bla)$), in Theorem 3.7 (iii),  such that $e\OO=\DD F_T^S$  then  we have: 
  \begin{enumerate}
\item
 \begin{equation}
  \displaystyle \OO = \bigoplus_{{T\in \STab(n)}}  \DD F_T= \bigoplus_{\bla \in \Pc_{r,n}} \bigg( \bigoplus_{{T\in \STab(\bla)}}  \DD F_T \bigg);
  \end{equation}
\item for each  $\lambda \in \Pc_{r,n}$  fix  an $r$-tableau $ T^* \in \STab(\lambda) $, then 
 \begin{equation}
  \displaystyle \OO =  \bigoplus_{\bla \in \Pc_{r,n}} f^\bla \DD F_{T^*}
  \end{equation}
  where $f^\bla =\dim_{\bC} (V(\lambda))$
\end{enumerate}

\end{proposition}
\begin{proof}
We have by the proof of  Theorem 3.7  that
$$ \OO = \bigoplus_{T_i \in \STab(n)} e_{i} \OO $$ and the $e_i \OO$ are simple $\DD$-modules. Since to each primitive idempotent $e_i$ corresponds an $r$-diagram  $\lambda_i \in \Pc_{r,n} $ and a tableau $T_i \in \STab(\lambda_i)$ such that $e_i\OO= \DD F_{T_i}$ then $\OO= \bigoplus_{T \in \STab(n)} \DD F_T.$ By Corollary 3.6, $ \DD F_{T_j} \cong  \DD F_{T_j}$ if $T_i,T_j \in \STab(\lambda)$ and so we have $f^\bla$ isomorphic copies of $ \DD F_{T^*}$ in the direct sum (3.1).
\end{proof}

 Using Proposition 3.1 and Proposition 3.2 we get the next theorem.
\begin{theorem}
\label{thm:one}
\begin{itemize}
\item[(i)] $N_T=\Dc_Y F_T$ is an irreducible $D_Y$-submodule of $\pi_+( \Oc_X)$.
\item[(ii)] There is a direct sum decomposition
\begin{equation}
\label{eq:direct sum} 
\pi_+ (\Oc_X)=\bigoplus_{\lambda \in \Pc_{r,n} }  \bigoplus_{T\in \STab(\lambda)}   N_T
\end{equation}

\item[(iii)] 
\begin{equation}
\label{eq:direct sum} 
\pi_+ (\Oc_X)\cong \bigoplus_{\lambda \in \Pc_{r,n} }  f^\la N_{T^*}
\end{equation}
\end{itemize}
\end{theorem}

We get in Theorm 3.10 a decomposition of the $\pi_+(\Oc_X)$ into irreducible $\Dc_Y$ modules generated by the higher Specht polynomials.

\section{Using correspondence between $G$-representations and $D$-modules}

In this section we us an equivalence of categories between  group representations and modules over differential ring to yield a better version of decomposition the polynomial ring as $\DD$-modules

\begin{lemma}
Let $\lambda \in \Pc_{r,n},$ $ S \in \STab(\lambda) $,  $ V_S(\lambda)$ be the corresponding irreducible  representation and $e_T $ the  primitive idempotent associated with an $r$-standard tableau $T\in \STab(\bla).$ Then $e_T(V_S({\lambda}))= \{ e_T (m) | m \in V_S({\lambda} )\}$  is a one dimensional $\bC$-vector space.
\end{lemma}
\begin{proof}
In fact we have 
\begin{eqnarray*}
e_ T V_S({\lambda}) &\cong&  e_T \bC[G(r,n)] e_T \\
&\cong&e_T \bC[G(r,n)] e_T \\
&\cong& \bC e_T \  \mbox{by \  \cite[ Chap III, \S 4, Theorem 3.9 ]{Boerner}}
\end{eqnarray*}

\end{proof}

Recall that if $M$ is a semi-simple module over a ring $R$, and $N$ is simple $R$-module, then the isotopic component of $M$ associated to $N$ is the sum $\sum N' \subset M$ of all $N'\subset M$ such that $N' \cong N.$

\begin{lemma}
Let $\lambda \in \Pc_{r,n},$ $ S \in \STab(\lambda) $ and  $ V_S(\lambda)$ be the corresponidng irreducible  module associated. Let $M:=\OO$ and $M({\lambda})$be the isotopic component of $M$ (as $\Oc_{V}$-module) associated with $V_S({\lambda})$. Then $M({\lambda})$ is $\DD $-module.
\end{lemma}
\begin{proof}
We only have to proof that $\DD \cdot M({\lambda}) \subset M({\lambda}).$ Let $ D \in \DD$ and $N$  be a $\bC[G(r,n)]$-module isomorphic to $V_S({\lambda})$, since $D$ commute with the elements of the group algebra $\bC[G(r,n)]$, $D$ is an $\bC[G(r,n)]$-homomorphism from $N$ into $D(N)$. Then by virtue of the Schur lemma $D(N)=0$ or $D(N)\cong N$ as a $\bC[G(r,n)]$-module, and $D(N) \subset M({\lambda})$. Hence $\DD \cdot M({\lambda}) \subset M({\lambda}).$
\end{proof}

\begin{lemma}
Let  $\lambda \in \Pc_{r,n},$ $ S \in \STab(\lambda)_d $, $M(\lambda)$ the isotopic component associated with  $V_S(\lambda)$ and $e_T $ the primitive idempotent associative with an $r$-standard tableau $T$ . Then
$e_T(M({\lambda}))$ si $\DD$-module.
\end{lemma}
\begin{proof}
Let  $D \in \DD,$ we have $ D(e_T (M^{\lambda})) = e_T (D(M({\lambda})) \subset e_T (M ({\lambda}))$, so $e_T(M ({\lambda}))$ si $\DD$-module.
\end{proof}

Before we proceed, let us recall the correspondence between $G$-representations and $D$-modules \cite[Paragraph 2.4]{IBK}.
Let $L$ and $K$ be  two extensions fields a field $k$, denote by $T_{K/k}$ the $k$-linear derivations of $K$. We  say that a $T_{K/k}$-module $M$
is  {\it $L$-trivial } if $L\otimes_KM \cong L ^n $ as
$T_{L/k}$-modules. Denote by  $\Mod^{L}(T_{K/k})$ the full subcategory of finitely generated $T_{K/k}$-modules that are $L$-trivial.
 It is immediate that it is closed under taking submodules and quotient modules. Using a lifting $\phi$ , $L$ may be thought of as a $T_{K/k}$-module. 
If $G$ is a finite group let $ \Mod (k[G]) $ be the category of finite-dimensional representations of $k[G]$.
Let now $k\to K \to L$ be a tower of fields such that $K= L^G$.  
Note that the action of $T_{K/k}$ commutes with the action of $G$. If $V$ is a $k[G]$-module,  $L\otimes_kV$ is a $T_{K/k}$-module by $D(l\otimes v)=D(l)\otimes v,\quad D\in T_{K/k}$, and  $(L\otimes_kV)^G$ is a $T_{K/k}$-submodule.

Let $L$ and $K$  be two fields, say that a $T_{K/k}$-module $M$
is  {\it $L$-trivial } if $L\otimes_KM \cong L ^n $ as
$T_{L/k}$-modules. Denote by  $\Mod^{L}(T_{K/k})$ the full subcategory of finitely generated $T_{K/k}$-modules that are $L$-trivial.
 It is immediate that it is closed under taking submodules and quotient modules. Using the lifting $\phi$ , $L$ may be thought of as a $T_{K/k}$-module. 
If $G$ is a finite group let $ \Mod (k[G]) $ be the category of finite-dimensional representations of $k[G]$.
Let now $k\to K \to L$ be a tower of fields such that $K= L^G$.  
Note that the action of $T_{K/k}$ commutes with the action of $G$. If $V$ is a $k[G]$-module,  $L\otimes_kV$ is a $T_{K/k}$-module by$D(l\otimes v)=D(l)\otimes v,\quad D\in T_{K/k}$, and  $(L\otimes_kV)^G$ is a $T_{K/k}$-submodule.

 \begin{proposition}
 \label{prop:equiv}
   The functor
   \begin{displaymath}
\nabla:     \Mod (k[G]) \to \Mod (T_{K/k}), \quad V\mapsto (L\otimes_k V)^G
   \end{displaymath}
   is fully faithful, and defines an equivalence of
   categories $$ \Mod (k[G]) \to \Mod^{L}(T_{K/k}).$$ The quasi-inverse
   of $\nabla$ is the functor 
   \begin{displaymath}
     Loc : \Mod^{L}(T_{K/k}) \to \Mod (k[G]),\quad Loc (M) = (L\otimes_K M)^{\phi(T_{K/k})}.
   \end{displaymath}
\end{proposition} 
\begin{proof}
see \cite[Proposition 2.4]{IBK}
\end{proof}

In the following proposition we take $G=G(r,n)$, $ K:=\bC(y_1, \ldots, y_n)$ the field of fractions of $\Oc_{V}$ and $L=\bC(x_1, \ldots, x_n)$ the field of fractions of $\OO$ so that $K=L^{G(r,n)}.$ It is clear that $L$ is a Galois extension of $K$ with Galois $G(r,n)$.

\begin{proposition}
Let  $\lambda \in \Pc_{r,n}$ $ S \in \STab(\lambda)$, $V_S(\lambda)$ be the corresponding irreducible  representation and $e_T $ the primitive idempotent  associative with an $r$-standard tableau $T$.
Set $M_T:=e_T \OO$.Then we have:
\begin{enumerate}

\item $M_T= \nabla(V_S({\lambda}))$
\item $M_T= e_T (M ({\lambda}))$ is simple $\DD$-module;
\item $M({\lambda}) =\displaystyle \bigoplus_{T \in \STab(\la)} e_T(M({\lambda})).$
\end{enumerate}
\end{proposition}
\begin{proof}
\begin{enumerate}

\item Let us consider  the right $\bC[G(r,n)]$-module $V=e_T \bC[G(r,n)]$ where $T \in \STab(\lambda)$. This is the image of $\bC[G(r,n)]$ by right multiplication map $e_T: \bC[G(r,n)] \to \bC[G(r,n)]$.  By \cite[Example 2.5]{IBK}, we may turn this map to a left multiplication $ \bC[G(r,n)]^r \to \bC[G(r,n)]^r$ and get an image which is isomorphic to $V_S({\lambda})$ for some $S\in \STab(\la)$. So we get an induced map
$$ \nabla (\bC[G(r,n)]^r ) \to \nabla (V_S({\lambda})) \subset \nabla (\bC[G(r,n)]^r ),$$ which is a multiplication by $e_T$ according to \cite[Example 2.5]{IBK}. Then $\nabla(V_S({\lambda}))$ is egal to $e_T \OO = M_T.$

\item Since $V_S({\lambda})$ is a simple $\bC[G(r,n)]$-module, $\nabla (V_S({\lambda}))$ is also a simple $\DD$-module.

\item follows from the fact that $1=\sum_{T \in \STab(n)} e_T$ and $e_T(M({\lambda}))=0 $ if $T \notin \STab(\bla)$ .
\end{enumerate}
\end{proof}

\begin{theorem}
Let $T$ be an $r$-standard tableau of shape $\lambda$ where $\lambda \in \Pc_{r,n}$ and $M_T$ be as in the above proposition. Then 
\begin{enumerate}
 \item $ M_T=\displaystyle \bigoplus_{ S\in \STab(\lambda)}\Oc_{V}  F_T^S$ as $\DD$-module, 
 
 \item $\OO = \displaystyle \bigoplus_{\lambda \in \Pc_{r,n}}  \bigg( \bigoplus_{ S, T\in \STab(\lambda)} \Oc_{V}  F_T^S \bigg) $ as a $\DD$-module.
 \end{enumerate}
 \end{theorem}
 
\begin{proof}
\begin{enumerate}
\item For $S \in \STab(\bla)$ We know by Theorem 2.2  that the polynomial $F_T^S$ generate a $\bC[G(r,n)]$-module inside $\OO$ which is isomorphic to $V({\lambda})$  . Then $F_T^S\in M^{\lambda}$ and $M^{\lambda}= \displaystyle \bigoplus_{S,T \in \STab(\lambda)} \bC[G(r,n)] F_T^S  \Oc_{V}$. 
Moreover $e_T( F_T^S)= c F_T^S, c\in \bC$ and by Lemma 4.1  $e_T (\bC[G(r,n)] F_T^S)=\bC F_T^S.$ Hence $M_T=e_T(M^{\lambda}) =\displaystyle \bigoplus_{S \in \STab(\lambda)}  \Oc_{V}  F_T^S.$
\item follows from  Proposition 3.9.
\end{enumerate}
\end{proof}

\begin{theorem}
Let $\lambda \in \Pc_{r,n}$ and ${\bf D} \in \DD$ such that ${\bf D}(F_T^S) \neq 0$ for a higher Specht polynomial $F_T^S$ with  $S, T \in \STab(\lambda)$. Then the image of the $\bC[G(r,n)]$-module $V_S({\lambda})$  by ${\bf D}$ is an $\bC[G(r,n)]$-module  isomorphic to $V_S({\lambda})$. In others words, the actions of the differential operators of $\DD$ on the higher Specht polynomials generate isomorphic copies of the corresponding irreducible $\bC[G(r,n)]$-module.
\end{theorem}

\begin{proof}
Let   $\lambda \in \Pc_{r,n}, {\bf D} \in \DD$ such that ${\bf D}(F_T^S) \neq 0$ for $S,T \in \STab(\lambda)$ and set $W^S_{\bf D}({\lambda})= {\bf D} (V_S({\lambda}))$ the image of the  module $V_S({\lambda})$  under the map ${\bf D}$. By Theorem 2.2, the $\bC$-vector space $V_S({\lambda})$ is  equipped with a basis $\C  = \{ F_T^S; T\in \STab (\lambda) \}$,  then $W^S_{\bf D}({\lambda})$ is the vector space  spanned by the set $\{ {\bf D}(F_T^S);\  T\in \STab(\lambda) \}$. The elements of $\{ F_T^S ; T\in \STab(\lambda) \}$ are linearly independent over $\DD$, otherwise the direct sum in  Theorem 4.7 cannot hold. It follows that the elements $\{ {\bf D}(F_T^S); T \in \STab(\lambda) \}$ are linear independent over $\bC$.  Hence $\{ {\bf D}(F_T^S); T \in \STab(\lambda) \}$ is a basis of $W^S_{\bf D}({\lambda})$ over $\bC$. Since ${\bf D}$ commute with  elements of $\bC[G(r,n)],\ W^S_{\bf D}({\lambda})$ is an $\bC[G(r,n)]$-module  isomorphic to $V_S({\lambda}).$
\end{proof}

\section*{Acknowledgments}
The final version work has been done while the first author was visiting IMSP at Benin,  under the  {\it Staff Exchange} program of the German  Office of Academic Exchange  (DAAD).
He warmly thanks   the DAAD for the financial support.


\end{document}